\documentclass[12pt]{amsart}

\usepackage{geometry} 
\usepackage{amssymb}
\usepackage{amscd}
\usepackage{amsfonts}
\usepackage{amsmath}
\usepackage{amssymb}
\usepackage[american, english]{babel}
\usepackage{bbm}
\usepackage{bookmark}
\usepackage{cmap}
\usepackage{dsfont}
\usepackage{enumerate}
\usepackage{epigraph}
\usepackage{euscript}
\usepackage[myheadings]{fullpage}
\usepackage{graphicx}
\usepackage{mathabx}
\usepackage{mathrsfs}
\usepackage{mathtools}
\usepackage{refcount}
\usepackage{stmaryrd}
\usepackage{thinsp}
\usepackage[safe]{tipa}
\usepackage{verbatim}
\usepackage{xr-hyper}
\usepackage{hyperref}
\usepackage[all]{xy}
\usepackage{url}

\newtheorem{definition}{Definition}

\newtheorem{lemma}{Lemma}

\newtheorem{expl}{Example}
\newtheorem{rmk}{Remark}
\newtheorem{theorem}{Theorem}
\newtheorem{corollary}{Corollary}
\newtheorem{q}{Question}
\newtheorem*{thm*}{Theorem}

\newcommand{\Spec}{\operatorname{Spec}}

\begin{document}

\title{Length of local cohomology in positive characteristic and ordinarity}

\author{Thomas Bitoun}

\thanks{Mathematical Institute, University of Oxford, Oxford OX2 6GG, UK; tbitoun@gmail.com}


\begin{abstract} 
Let $D$ be the ring of Grothendieck differential operators of the ring $R$ of polynomials in $d\geq3$ variables with coefficients in a perfect field of characteristic $p.$ We compute the $D$-module length of the first local cohomology module $H^1_f(R)$ with respect to a polynomial $f$ with an isolated singularity, for $p$ large enough. The expression we give is in terms of the Frobenius action on the top coherent cohomology of the exceptional fibre of a resolution of the singularity. Our proof rests on a tight closure computation of Hara. Since the above length is quite different from that of the corresponding local cohomology module in characteristic zero, we also consider a characteristic zero $D$-module whose length is expected to equal that above, for ordinary primes.
\end{abstract}

\maketitle

\section{Introduction}

In this note, we compute the positive characteristic $D$-module length of the first local cohomology module of the structure sheaf with support in a hypersurface, in a large class of examples. 
Our main result can also be seen as part of our study of the $b$-function in positive characteristic, see \cite{bfunp}. On the one hand, in \cite{bfunp} using $D$-module (or unit $F$-module) techniques, for $D$ the ring of Grothendieck differential operators, we associate to a non-constant polynomial $f$ with coefficients in a perfect field of positive characteristic a set of $p$-adic integers, called the roots of the $b$-function of $f.$ On the other, one may consider the $F$-jumping exponents of the generalised test ideals of $f,$ see \cite{MR1974679}. These are positive real numbers which are characterised by their intersection with the unit interval $(0,1]$ and have been shown to be rational numbers in \cite{MR2538604}. In \cite{bfunp} we prove that the roots of the $b$-function of $f$ are exactly the opposites of the $F$-jumping exponents of $f$ which are in $\mathbb{Q}\cap \mathbb{Z}_p.$ It would thus seem that the information provided by the $F$-jumping exponents of $f$ which are not in $\mathbb{Q}\cap \mathbb{Z}_p,$ i.e. whose denominator is divisible by $p,$ let us call them irregular, is lost in the theory. 
A consequence of the results presented here is that not all the information is lost. Namely the absence of irregular $F$-jumping exponents is well-known to be closely related to phenomena of ordinarity, see \cite{MR2863367}. We claim that at the very least the $D$-module (or unit $F$-module) length of the module $N_f$ used to define the $b$-function in \cite{bfunp} distinguishes ordinary primes from supersingular ones, for large enough primes. More precisely, using the terminology of \cite{bfunp}, one can see that the joint eigenspace of the action of the higher Euler operators on $N_f$ corresponding to the root {-1} of the $b$-function of $f$ is isomorphic to the first local cohomology module $H^1_f(R)= \frac{R[\frac{1}{f}]}{R},$ where $R$ is the ring of polynomials. Since this assertion does not appear in the literature, let us give a proof.  
By \cite[Proposition 1]{bfunp},  
this joint eigenspace is the limit of the direct system: $\frac{D^{(0)}f^{p-1}}{D^{(0)}f^p} \xrightarrow{.f^{p^2-p}} \frac{D^{(1)}f^{p^2-1}}{D^{(1)}f^{p^2}} \xrightarrow{.f^{p^3-p^2}} \cdots \xrightarrow{.f^{p^{e+1}-p^{e}}} \frac{D^{(e)}f^{p^{e+1}-1}}{D^{(e)}f^{p^{e+1}}} \xrightarrow{.f^{p^{e+2}-p^{e+1}}} \frac{D^{(e+1)}f^{p^{e+2}-1}}{D^{(e+1)}f^{p^{e+2}}} \xrightarrow{.f^{p^{e+3}-p^{e+2}}}\cdots,$ where $D^{(e)}$ is the ring of Grothendieck differential operators of level $e$ on $\Spec(R).$
Since $D^{(e)}f^{p^{e+1}-1} = R$ for every large enough natural number $e$ by \cite[Proof of Lemma 6.8]{MR2469353}, and the $D$-submodule $D\frac{1}{f}$ of $R[\frac{1}{f}]$ generated by $\frac{1}{f}$ is $R[\frac{1}{f}]$ itself by \cite[Theorem 1.1]{MR2155224}, the assertion immediately follows from the description of $D\frac{1}{f}$ given in \cite[Remark 2.7]{MR2538604}.

For a $d$-dimensional proper variety $Z$ over a field of characteristic $p>0,$ we let the $p$-genus $g_p(Z)$ of $Z$ be the dimension of the stable part \cite{MR0441962} of $\overline{k}\otimes H^d(Z, \mathcal{O}_{Z}),$ that is $\dim_{\overline{k}}(\cap_{l\geq0}F^l(\overline{k}\otimes H^d(Z, \mathcal{O}_{Z}))),$ where $F$ is the Frobenius action on coherent cohomology and $\overline{k}$ is an algebraic closure of the base field.
 Our main result is (see Theorem \ref{thm: main} for the precise general formulation): 

\begin{theorem}
Suppose that $f$ is an irreducible complex polynomial in $n\geq3$ variables with an isolated singularity at the origin and let $Y\xrightarrow{\pi} X$ be a resolution of the singularity. Then for almost all $p,$ the $D$-module length of $H^1_{f_p}(R)$ is $1 + g_p(Z_p),$ where $Z_p$ is the reduction modulo $p$ of the exceptional fibre of $\pi.$ \end{theorem}

The proof, which mostly belongs to the theory of unit $F$-modules, uses Blickle's intersection homology $D$-module \cite{MR2036330} and Lyubeznik's enhancement of Matlis duality \cite{MR1476089} to reduce the main unit $F$-module length computation to a geometric description of the tight closure of $0$ in the local cohomology of the singularity, due to Hara \cite{MR1646049}. One then deduces the $D$-module length from Blickle's length comparison result \cite{MR1946409} and an application of Haastert's positive characteristic Kashiwara's equivalence \cite{MR966631}.  

We note that in characteristic zero, the $D$-module length of the first local cohomology module is of a quite different nature. It actually is a topological invariant. For example, let $f$ be a rational cubic in three variables which is the equation of an elliptic curve $E$ in $\mathbb{P}^2_{\mathbb{Q}},$ of genus $g=1.$ Let $R_{\mathbb{C}}$ be the ring of complex polynomials in $3$ variables and for all primes $p,$ let $R_p$ be the ring of $\mathbb{F}_p$-polynomials in $3$ variables. Then the $D_{R_{\mathbb{C}}}$-module length of the local cohomology module $H^1_f(R_{\mathbb{C}})$ is $3= 1+2g,$ see e.g. \cite[Remark 1.2]{bitoun2016d}. But as will be seen in Example \ref{expl: elliptic curve}, for almost all primes $p,$ the $D_p$-module length of $H^1_{f_p}(R_p)$ is $2= 1+g$ if $E_p$ is ordinary and $1$ if $E_p$ is supersingular, where $E_p$ (resp. $f_p$) is the reduction of $E$ (resp. $f$) modulo $p$ and $D_p= D_{R_{\mathbb{F}_p}}$ is the ring of Grothendieck differential operators on $\mathbb{A}^{n}_{\mathbb{F}_p}.$ Thus for (almost) all primes $p,$ the lengths of the first local cohomology modules in characteristic zero and in characteristic $p$ are different.  
We end this note with a section on comparison with characteristic zero, arguing that in great generality, the $D_{R_{\mathbb{C}}}$-submodule $D_{R_{\mathbb{C}}}\frac{1}{f}$ of the local cohomology module $H^1_f(R_{\mathbb{C}})$ generated by the class of $\frac{1}{f}$ (which need not be equal to $H^1_f(R_{\mathbb{C}})$) is a better behaved characteristic zero analogue of $H^1_{f_p}(R_p)$ than the whole local cohomology module $H^1_f(R_{\mathbb{C}}).$ (Recall that the left $D_p$-module $H^1_{f_p}(R_p)$ is generated by the class of $\frac{1}{f_p},$ by \cite[Theorem 1.1]{MR2155224}.)
For example, in the case of the elliptic curve above, we have that the $D_{R_{\mathbb{C}}}$-module length of $D_{R_{\mathbb{C}}}\frac{1}{f}$ is equal to the $D_p$-module length of $H^1_{f_p}(R_p),$ for almost all ordinary primes $p$ of $E.$ The $D_{R_{\mathbb{C}}}$-module $D_{R_{\mathbb{C}}}\frac{1}{f}$ is studied in detail in \cite{bitoun2016d}. (See also \cite{saito2015d} for a different approach.) 

\subsection{Notation}
Throughout the note we will use the following notation: For an integer $n\geq2$ and all fields $K,$ we let $R_K$ be the ring of polynomials in $n+1$ variables $\{x_0, \dots, x_n\}$ with coefficients in $K$ and $D_K = D_{R_K}$ be the ring of Grothendieck differential operators on $\mathbb{A}^{n+1}_K.$ 

Let $k$ be a perfect field of positive characteristic $p,$ we set $D = D_{R_k}.$ If $A$ is a $k$-algebra, we denote by $A[F]$ the twisted polynomial ring over $A$ whose multiplication is defined by $Fa=a^pF,$ for all $a$ in $A.$  If $M$ is a left $A[F]$-module, we denote by 
$F^*M \xrightarrow{\beta_M} M$ the $A$-linear morphism induced by the action of $F$ on $M,$ where $F^*$ is the functor from the category of $A$-modules to itself, given by the extension of scalars by the Frobenius endomorphism of $A.$

\section{Length of the First Local Cohomology in Positive Characteristic} 
We first recall some definitions. 

\begin{definition} Let $f\in R_k$ be a non-constant polynomial in $n+1$ variables. The first local cohomology module $H_f^1(R_k)$ of $R_k$ with respect to $f$ is the left $D$-module cokernel of the natural inclusion $R_k\subset R_k[\frac{1}{f}].$ \end{definition}

\begin{rmk} The Frobenius endomorphism of $R_k$ induces a finitely generated unit $F$-module structure on $H_f^1(R_k).$ The associated action of $D$ is the natural one. Hence 
it follows immediately from \cite[Theorem 3.2]{MR1476089} that $H_f^1(R_k)$ is of finite length as a unit $F$-module. It is thus of finite length as a left $D$-module by \cite[Theorem 5.7]{MR1476089}.
\end{rmk}

The purpose of this note is to give an expression for the length of $H_f^1(R_k),$ when $f$ has an isolated singularity. It will be in terms of the quasilength of a certain $F$-module. We now recall the definitions from \cite[Section 4]{MR1476089}. 

\begin{definition} \label{defi: *}
Let $A$ be a local Noetherian $k$-algebra with Frobenius endomorphism $F$ and let $M$ be a left $A[F]$-module. 
\begin{itemize}
\item $M^\ast:= \cap_{n\geq0} F^nM,$ where $F^nM$ is the $A$-submodule of $M$ generated by the image $F^n(M)$
\item $M_{nil}:= \cup_{n\geq0} \ker \{M \xrightarrow{F^n} M\}$
\end{itemize}
\end{definition}

\begin{definition} Let $A$ be a local Noetherian $k$-algebra of Frobenius endomorphism $F$ and let $M$ be a left $A[F]$-module. Suppose that $M$ is Artinian as an $A$-module.
\begin{itemize} \item A finite chain of length $s$ of $A[F]$-submodules $0= M_0 \subset \dots \subset M_s=M$ is quasimaximal if $(\frac{M_i/M_{i-1}}{(M_i/M_{i-1})_{nil}})^\ast$ is a simple left $A[F]$-module, for all $i \in \{1, \dots, s\}.$
\item If $M_{nil} \subsetneq M,$ then $M$ has a quasimaximal chain of submodules and all such chains are of the same length, called the quasilength $\mathrm{ql}(M)$ of $M.$ If $M= M_{nil},$ then we set $\mathrm{ql}(M)=0.$ See \cite[Theorems 4.5 and 4.6]{MR1476089}.
\end{itemize}
\end{definition}

Finally, we recall the definition of the Lyubeznik-Matlis duality.

\begin{definition} \label{defi:LM duality} Let $\widehat{R}_0$ be a complete local regular Noetherian $k$-algebra and let $M$ be an Artinian $\widehat{R}_0$-module.
\begin{itemize}
\item  
The Matlis duality functor is the contravariant functor $(-)^\vee:= Hom_{\widehat{R}_0}(-, E)$ from the category of $\widehat{R}_0$-modules to itself, 
where $E$ is an injective hull of the residue field of $\widehat{R}_0$ in the category of $\widehat{R}_0$-modules. Moreover, the Matlis dual $M^\vee$ of an Artinian module $M$ is a finitely generated $\widehat{R}_0$-module.

\item Suppose further that $M$ is a left $\widehat{R}_0[F]$-module. The Lyubeznik-Matlis dual $\mathcal{D}(M)$ of $M$ is the finitely generated unit $\widehat{R}_0[F]$-module given by the 
direct limit of the following direct system in the category of $\widehat{R}_0$-modules:
$M^\vee \xrightarrow{{\beta}^\vee_M} F^*(M^\vee) \xrightarrow{F^*{\beta}^\vee_M} {F^*}^2(M^\vee) \to \dots,$ where we have used the canonical isomorphism $(F^*M)^\vee \cong F^*(M^\vee)$ of \cite[Lemma 4.1]{MR1476089}. Lyubeznik-Matlis duality $\mathcal{D}$ is a contravariant functor from the category of left $\widehat{R}_0[F]$-modules which are Artinian as $\widehat{R}_0$-modules, to the category of finitely generated unit $\widehat{R}_0[F]$-modules.
\end{itemize}\end{definition}

To state our main result, we need to introduce the following notation: 

Let $L$ be a field of characteristic $0$ and let $g$ be a non-constant polynomial in $n+1$ variables $\{x_0, \dots, x_n\},$ with coefficients in $L.$ Let $(A, \frak{m})$ be the local ring of the zero-locus of $g$ at a singular point $z.$ Let us fix a resolution $X \xrightarrow{\pi} Z= \text{Spec}(A)$ of the singularity $z$ and let $Y$ be the fibre of $\pi$ at $z.$

\begin{definition}
Let $B\subset L$ be a finitely generated subring, containing $1.$ We say that $B$ is a ring of definition of $\pi$ if the coefficients of $g$ are contained in $B$ and there is a resolution of singularities of $B$-schemes $X_B \xrightarrow{\pi_B} Z_B$ whose base-change $L\otimes_{Frac(B)} \pi_B$ is isomorphic to $\pi.$ \end{definition}
For each closed point $u$ of $\text{Spec}(B),$ we let $X_u \xrightarrow{\pi_u} Z_u$ (resp. $g_u,$ resp. $Y_u$) be the fibre of $\pi_B$ (resp. $g,$ resp. $Y$) over $k(u).$ Finally, we consider the coherent cohomology groups $H^l(X_u, \mathcal{O}_{X_u})$ (resp. $H^l(Y_u, \mathcal{O}_{Y_u})$ ) as left $A[F]$-modules (resp. $k(u)[F]$-modules) for the action of the Frobenius endomorphism on the cohomology. Here is our main result:

\begin{theorem} \label{thm: main}
Suppose that $n\geq2$ and that $g$ is absolutely irreducible with an isolated singularity at the origin. Then there is a ring of definition $B \subset L$ of $\pi$ such that, for all closed points $u$ of $\text{Spec}(B):$ 

\begin{itemize}
\item[(i)] \label{item: uF-length}
The unit $F$-module length of the first local cohomology group $H_{g_u}^1(R_{k(u)})$ is 

$1 + \mathrm{ql}(H^{n-1}(X_u, \mathcal{O}_{X_u})) = 1 +  \mathrm{ql}(H^{n-1}(Y_u, \mathcal{O}_{Y_u})).$
\item[(ii)] \label{item: D-length}
The $D_{k(u)}$-module length of $H_{g_u}^1(R_{k(u)})$ is 
$$1 + \mathrm{ql}(\overline{k(u)}\otimes H^{n-1}(X_u, \mathcal{O}_{X_u})) = 1 + \dim_{\overline{k(u)}}((\overline{k(u)}\otimes H^{n-1}(Y_u, \mathcal{O}_{Y_u}))^\ast),$$ where $\overline{k(u)}$ is any algebraic closure of $k(u)$ and $(-)^\ast$ is the operation on $\overline{k(u)}[F]$-modules from Definition \ref{defi: *}.
\end{itemize}
\end{theorem}

\begin{proof} 
By Ostrowski's Theorem, see \cite[Lemma 11]{MR0257033} for a quick proof, there is a definition ring $B'$ of $\pi$ such that, for all closed points $u$ of $\text{Spec}(B'),$ $g_u$ is absolutely irreducible. 

For every closed point $u$ of $\text{Spec}(B'),$ we will use the following notation: $(A_u, \mathfrak{m}_u)$ is the local ring of the singularity, $(R_{0}, \mathfrak{m}):= ((R_{k(u)})_{(x_0,\dots, x_n)}, (x_0,\dots, x_n)(R_{k(u)})_{(x_0,\dots, x_n)})$ and $(R_{\overline{0}}, \overline{\mathfrak{m}}):= ((R_{\overline{k(u)}})_{(x_0,\dots, x_n)},(x_0,\dots, x_n)(R_{\overline{k(u)}})_{(x_0,\dots, x_n)}).$ We denote their completion with respect to their maximal ideal by $(\widehat{A}_0, \widehat{\mathfrak{m}_0}) , (\widehat{R}_0, \widehat{\mathfrak{m}})$ and $(\widehat{R_{\overline{0}}}, \widehat{\overline{\mathfrak{m}}}),$ respectively.

We have a short exact sequence of both $D_{k(u)}$- and unit $F$-modules:
\begin{equation} \label{equation: ses}
0\to \mathcal{L} \to H_{g_u}^1(R_{k(u)}) \to \mathcal{K} \to 0
\end{equation}
where $\mathcal{L}$ is the intersection homology module $\mathcal{L}(\mathbb{A}^{n+1}_{k(u)}, \{g_u = 0\})$ of \cite{MR2036330} and $\mathcal{K}$ is supported at the origin.
Tensoring with the completion $\widehat{R}_0$ of the local ring at the origin, we get a short exact sequence: $$0\to \widehat{R}_0\otimes_{R_{k(u)}} \mathcal{L} \to \widehat{R}_0\otimes_{R_{k(u)}} H_{g_u}^1(R_{k(u)}) \to 
\widehat{R}_0\otimes_{R_{k(u)}}\mathcal{K} \to 0$$ which we can rewrite as
$0\to \mathcal{L}'\to H^1_{g_u}(\widehat{R}_0) \to\mathcal{K}'\to 0,$ where 
$\mathcal{L}'= \mathcal{L}(\frac{\widehat{R}_0}{g_u\widehat{R}_0}, \widehat{R}_0)$ and $\mathcal{K}'=  \widehat{R}_0\otimes_{R_{k(u)}}\mathcal{K}.$ Indeed  $\mathcal{L}'\cong \widehat{R}_0\otimes_{R_{k(u)}}\mathcal{L}$ by \cite[Theorem 4.6]{MR2036330} and it is well-known that local cohomology commutes with base-change by the completion.
Clearly the length of $\mathcal{K}'$ as a $D_{\widehat{R}_0}$-module (resp. unit $F$-module) equals the length of $\mathcal{K}$ as a $D_{k(u)}$-module (resp. unit $F$-module). Hence so is the case for $H_{g_u}^1(R_{k(u)})$
and $H^1_{g_u}(\widehat{R}_0),$ since $\mathcal{L}$ and $\mathcal{L}'$ are irreducible.

Let $D':= D_{\widehat{R}_0}.$ We also let $\lg_{uF}(-)$ be the unit $F$-module length and $\lg_{D'}(-)$ to be the $D'$-module length. 
The proof of (i) thus reduces to: There exists a ring of definition $B\supset B'$ of $\pi$ such that $\Spec(B)\subset \Spec(B')$ is a dense open subset and, for all closed points $u$ of $\Spec(B),$
$\lg_{uF}(\mathcal{K}')= \mathrm{ql}(H^{n-1}(X_u, \mathcal{O}_{X_u})).$ Let us prove this assertion.

We will use the notation of \cite{MR2036330}. Using Lyubeznik-Matlis duality $\mathcal{D},$ see Definition \ref{defi:LM duality}, for all closed points $u$ of $\text{Spec}(B'),$ we have $H^1_{g_u}(\widehat{R}_0) \cong \mathcal{D}(H^n_{\widehat{\mathfrak{m}}}(\widehat{A}_0))$ by \cite[Proposition 2.16]{MR2036330}. By \cite[Theorem 4.4]{MR2036330}, we also have $\mathcal{L}' \cong \mathcal{D}(\frac{H^n_{\widehat{\mathfrak{m}}}(\widehat{A}_0)}{0^\ast_{H^n_{\widehat{\mathfrak{m}}}}(\widehat{A}_0)}),$  where $0^\ast_{H^n_{\widehat{\mathfrak{m}}}(\widehat{A}_0)} \subset H^n_{\widehat{\mathfrak{m}}}(\widehat{A}_0)$ is the tight closure of zero. Since Lyubeznik-Matlis duality exchanges unit $F$-module length with quasilength by the proof of \cite[Theorems 4.5]{MR1476089}, we have $\lg_{uF}(H^1_{g_u}(\widehat{R}_0)) = 1 + \mathrm{ql}(0^\ast_{H^n_{\widehat{\mathfrak{m}}}(\widehat{A}_0)}).$ Moreover by Lemma \ref{lemm: tight closure} applied to $R = R_{0}$ and $M = A_u,$ $\mathrm{ql}(0^\ast_{H^n_{\widehat{\mathfrak{m}}}(\widehat{A}_0)})$ is equal to $\mathrm{ql}(0^\ast_{H^n_{\mathfrak{m}}(A_u)}).$

Finally, since $A$ is an isolated singularity and $n\geq2,$ it is normal. Hence by \cite[Theorem 4.7]{MR1646049}, there exists a ring of definition $B\supset B'$ of $\pi$ such that $\Spec(B)\subset \Spec(B')$ is a dense open subset and, for all closed points $u$ of $\Spec(B),$ $0^\ast_{H^n_{\mathfrak{m}}(A_u)} \cong H^{n-1}(X_u, \mathcal{O}_{X_u}),$ as $A_u[F]$-modules. But by Lemma \ref{lemm: quasilength}, $\mathrm{ql}(H^{n-1}(X_u, \mathcal{O}_{X_u})) = \mathrm{ql}(H^{n-1}(Y_u, \mathcal{O}_{Y_u}))$ as $H^{n-1}(Y_u, \mathcal{O}_{Y_u})\cong\frac{H^{n-1}(X_u, \mathcal{O}_{X_u})}{\mathfrak{m}_uH^{n-1}(X_u, \mathcal{O}_{X_u})}$ by Lemma \ref{lemm: cd}.
This concludes the proof of (i).

We now prove (ii). 
From (\ref{equation: ses}), we deduce the short exact sequence $$0\to \overline{k(u)}\otimes\mathcal{L} \to H_{g_u}^1(R_{\overline{k(u)}}) \to \overline{k(u)}\otimes\mathcal{K} \to 0$$
Therefore, tensoring with the completion $\widehat{R_{\overline{0}}}$ of $R_{\overline{0}},$ we also have the short exact sequence
$$0\to \overline{k(u)}\otimes\mathcal{L}' \to H_{g_u}^1(\widehat{R_{\overline{0}}}) \to \overline{k(u)}\otimes\mathcal{K}' \to 0,$$
with $\mathcal{L}'$ and $\mathcal{K}'$ as above.
Note that $\overline{k(u)}\otimes\mathcal{L}'= \mathcal{L}(\frac{\widehat{R_{\overline{0}}}}{g_u\widehat{R_{\overline{0}}}}, \widehat{R_{\overline{0}}})$ by \cite[Lemma 5.16]{blickle2001intersection}. 
Moreover, since the injective hull $H^{n+1}_{\widehat{\mathfrak{m}}\otimes \overline{k(u)}}(\widehat{R_{\overline{0}}})$ of $\overline{k(u)}=\frac{\widehat{R_{\overline{0}}}}{\widehat{\mathfrak{m}}\otimes \overline{k(u)}}$ is isomorphic to $H^{n+1}_{\widehat{\mathfrak{m}}}(\widehat{R}_0)\otimes \overline{k(u)},$ it is easy to see that Matlis duality commutes with the field extension $-\otimes_{k(u)}\overline{k(u)}.$
Hence Lyubeznik-Matlis duality commutes with $-\otimes_{k(u)}\overline{k(u)}.$ Thus we have that $\overline{k(u)}\otimes\mathcal{L}'\cong \overline{k(u)}\otimes\mathcal{D}(\frac{H^n_{\widehat{\mathfrak{m}}}(\widehat{A}_0)}{0^\ast_{H^n_{\widehat{\mathfrak{m}}}(\widehat{A}_0)}}) \cong \mathcal{D}(\frac{H^n_{\overline{k(u)}\otimes\mathfrak{m}}(\overline{k(u)}\otimes\widehat{A}_0)}{\overline{k(u)}\otimes 0^\ast_{H^n_{\widehat{\mathfrak{m}}}(\widehat{A}_0)}}).$ Therefore the length of the unit $F$-module $H_{g_u}^1(\widehat{R_{\overline{0}}})$ is equal to $1+ \mathrm{ql}(\overline{k(u)}\otimes 0^\ast_{H^n_{\widehat{\mathfrak{m}}}(\widehat{A}_0)})= 1+ \mathrm{ql}(\overline{k(u)}\otimes H^{n-1}(X_u, \mathcal{O}_{X_u})).$
Also, similarly as above, the unit $F$-module length of $H_{g_u}^1(R_{\overline{k(u)}})$ is equal to the length of $H_{g_u}^1(\widehat{R_{\overline{0}}}),$ as a unit $F$-module. We thus deduce that the length of the unit $F$-module $H_{g_u}^1(R_{\overline{k(u)}})$ is $1+ \mathrm{ql}(\overline{k(u)}\otimes H^{n-1}(X_u, \mathcal{O}_{X_u})).$ Moreover $\mathrm{ql}(\overline{k(u)}\otimes H^{n-1}(X_u, \mathcal{O}_{X_u}))= \mathrm{ql}(\overline{k(u)}\otimes H^{n-1}(Y_u, \mathcal{O}_{Y_u}))$ by Lemmas \ref{lemm: quasilength} and \ref{lemm: cd}, and $\mathrm{ql}(\overline{k(u)}\otimes H^{n-1}(Y_u, \mathcal{O}_{Y_u}))= \dim_{\overline{k(u)}}((\overline{k(u)}\otimes H^{n-1}(Y_u, \mathcal{O}_{Y_u}))^\ast)$ by Lemma \ref{lemm: ql^*}.

But by \cite[Theorem 1.1]{MR1946409}, the length of $H_{g_u}^1(R_{\overline{k(u)}})$ as a unit $F$-module is equal to its length as a $D_{\overline{k(u)}}$-module. Finally, we claim that the $D_{\overline{k(u)}}$-module length of $H_{g_u}^1(R_{\overline{k(u)}})$ is equal to the $D_{k(u)}$-module length of $H_{g_u}^1(R_{k(u)}).$ This implies part (ii) of the theorem. 

Let us prove this last claim. We let $\lg_{D_{\overline{k(u)}}}(-)$ (resp. $\lg_D(-)$) denote the $D_{\overline{k(u)}}$-module (resp. $D_{k(u)}$-module) length. Localising at the origin, one sees by \cite[Lemma 5.16]{blickle2001intersection} that $\overline{k(u)}\otimes\mathcal{L}$ is the intersection homology module. Thus $\lg_{D_{\overline{k(u)}}}(\overline{k(u)}\otimes\mathcal{L})=1= \lg_D(\mathcal{L}).$ Hence the claim reduces to the equality $\lg_{D_{\overline{k(u)}}}(\overline{k(u)}\otimes\mathcal{K})= \lg_D(\mathcal{K}).$ But this follows immediately from the compatibility with base-field extension of Kashiwara's equivalence, see \cite[Corollary 8.13]{MR966631} (the proof of which is well-known to be valid over an arbitrary field of positive characteristic). Indeed by Kashiwara's equivalence, we have $\lg_{D_{\overline{k(u)}}}(\overline{k(u)}\otimes\mathcal{K})= \dim_{\overline{k(u)}}(\overline{k(u)}\otimes V)= \dim_{k(u)}(V)= \lg_D(\mathcal{K}),$ for a certain finite dimensional $k(u)$-vector space $V.$ 
\end{proof}

Recall that a ring $R$ is $F$-finite if it is  of positive characteristic, Noetherian and if the Frobenius map $\Spec(R) \xrightarrow{F} \Spec(R)$ is finite.

\begin{lemma} \label{lemm: tight closure}
Let $(R, \mathfrak{m})$ be an $F$-finite regular local ring and let $M$ be a finitely generated $R$-module. Then for all $i\geq0,$
the canonical isomorphism $H^i_{\mathfrak{m}}(M) \widetilde{\to} H^i_{\widehat{\mathfrak{m}}}(\widehat{M})$ induces an isomorphism of tight closures $0^\ast_{H^i_{\mathfrak{m}}(M)}\cong0^\ast_{H^i_{\widehat{\mathfrak{m}}}(\widehat{M})},$
where $(\widehat{R},\widehat{\mathfrak{m}})$ (resp. $\widehat{M}$) is the $\mathfrak{m}$-adic completion of $(R, \mathfrak{m})$ (resp. $M$).
\end{lemma}

\begin{proof} The isomorphism $H^i_{\mathfrak{m}}(M) \widetilde{\to} H^i_{\widehat{\mathfrak{m}}}(\widehat{M})$ is well-known, see \cite[Proposition 2.15]{MR2355770}. Furthermore the existence of completely stable big test elements for $R$ (\cite[Theorem p.77]{hochster2007foundations}) immediately implies the equality of tight closures.
\end{proof}

\begin{lemma} \label{lemm: ql^*}
Let $\overline{k}$ be an algebraically closed field of positive characteristic $p$ and let $V$ be a $\overline{k}$-finite dimensional left $\overline{k}[F]$-module. Then the quasilength of $V$ is $\dim_{\overline{k}}(V^\ast),$ where $(-)^\ast$ is the operation on $\overline{k}[F]$-modules from Definition \ref{defi: *}.\end{lemma}

\begin{proof} By definition of quasilength, we have $\mathrm{ql}(V)= \mathrm{ql}(V^\ast).$ Moreover, $F$ acts surjectively and thus injectively on $V^\ast.$ Hence, by \cite[Proposition 4.6]{MR1946409} for example, $V^\ast$ has a $\overline{k}$-basis of vectors fixed by $F.$ The lemma easily follows. 
\end{proof}

\begin{lemma}\label{lemm: quasilength}
Let $(A, \mathfrak{m})$ be a local Noetherian $k$-algebra with Frobenius endomorphism $F$ and let $M$ be a left $A[F]$-module. Suppose that $M$ is Artinian and Noetherian as an $A$-module. If $M$ is supported at the maximal ideal $\mathfrak{m},$ then $M$ and $\frac{M}{\mathfrak{m}M}$ have the same quasilength.
\end{lemma}

\begin{proof} Let us show that $\mathfrak{m}M\subset M_{nil}.$ This implies the lemma since $\mathrm{ql}(M_{nil})=0.$

Since $M$ is supported at $\mathfrak{m}$ and is Noetherian, $\mathfrak{m}^lM=0$ for some $l\geq0.$ Thus $F^r(\mathfrak{m}M) \subset \mathfrak{m}^{p^r}M=0$ for some $r\geq0.$ Hence $\mathfrak{m}M\subset M_{nil},$ as claimed.
\end{proof}

\begin{lemma} \label{lemm: cd}
Let $(A, \frak{m})$ be a Noetherian local ring and let $X \xrightarrow{\pi} Z = \Spec(A)$ be a projective morphism of special fibre $Y.$ Suppose that the fibres of $\pi$ are of dimension at most $d.$ Then for all quasi-coherent sheaves $\mathcal{F}$ on $X,$ the canonical morphism $H^d(X, \mathcal{F}) \to H^d(Y, \frac{\mathcal{F}}{\frak{m}\mathcal{F}})$ induces an isomorphism $\frac{H^d(X, \mathcal{F})}{\frak{m}H^d(X, \mathcal{F})} \to H^d(Y, \frac{\mathcal{F}}{\frak{m}\mathcal{F}}).$
\end{lemma}

\begin{proof} We first claim that for all quasi-coherent sheaves $\mathcal{F}$ on $X$ and for all integers $i\geq d+1, H^i(X, \mathcal{F})=0.$ This is well-known. 
Indeed it follows from the theorem on formal functions that $R^i\pi_*(\mathcal{G})=0$ for all $i\geq d+1,$ and for all coherent sheaves $\mathcal{G}$ on $X$ (\cite[Corollary III 11.2]{MR0463157}). Thus, since $Z$ is affine, we have that $H^i(X, \mathcal{G})=0,$ for all $i\geq d+1,$ and for all coherent sheaves $\mathcal{G}$ on $X.$ Since a quasi-coherent sheaf is the union of its coherent subsheaves, the claim then follows from the commutation of cohomology with direct limits (\cite[Proposition III 2.9]{MR0463157}).

Let $\{r_1, \dots, r_N\}$ be a set of generators of the maximal ideal $\frak{m}.$ We have the following short exact sequences of quasi-coherent sheaves on $X:$ $0\to\frak{m}\mathcal{F} \to \mathcal{F} \to \frac{ \mathcal{F}}{\frak{m}\mathcal{F}} \to 0$ and $0 \to J \to \oplus^{l=N}_{l=1} \mathcal{F} \xrightarrow{\phi} \frak{m}\mathcal{F} \to 0,$ where $\phi(f_1, \dots, f_N)= r_1f_1 + \dots + r_Nf_N$ and $J$ is the kernel of $\phi.$ Moreover by the claim, we have that $H^{d+1}(X, \frak{m}\mathcal{F}) = H^{d+1}(X, J)=0.$ Thus in the associated long exact sequences in Zariski cohomology, the morphisms $H^d(X, \mathcal{F}) \to H^d(X, \frac{ \mathcal{F}}{\frak{m}\mathcal{F}})$ and $\oplus^{l=N}_{l=1} H^d(X, \mathcal{F}) \to H^d(X, \frak{m}\mathcal{F})$ are surjective. The latter implies that the image of the morphism $H^d(X, \frak{m}\mathcal{F}) \to H^d(X, \mathcal{F})$ from the long exact sequence is $\frak{m}H^d(X, \mathcal{F}).$ This concludes the proof of the lemma.
\end{proof}

If $g$ is homogeneous, Theorem \ref{thm: main} may be rephrased without mentioning a resolution of the singularity. Let $Y$ be the hypersurface defined by $g$ in $\mathbb{P}^n.$ We first fix the notation.  

\begin{definition}
Let $B\subset L$ be a finitely generated subring, containing $1.$ We say that $B$ is a ring of definition of $Y$ if the coefficients of $g$ are contained in $B$ and there is a smooth projective hypersurface $Y_B$ of $\mathbb{P}^n_B$ whose base-change $L\otimes_{Frac(B)} Y_B$ is isomorphic to $Y.$ \end{definition}

Given such an hypersurface $Y_B,$ for each closed point $u$ of $\text{Spec}(B),$ we let $Y_u$ be the fibre of $Y_B$ over $k(u).$ Here is the result: 

\begin{corollary} \label{cor: homogeneous}
Under the same hypotheses as in Theorem \ref{thm: main}, assume that $g$ is homogeneous and let $Y$ be the hypersurface defined by $g$ in $\mathbb{P}^n.$ Then there is a ring of definition $B \subset L$ of $Y$ such that, for all closed points $u$ of $\text{Spec}(B):$ 
\begin{itemize}
\item[(i)] 
The unit $F$-module length of $H_{g_u}^1(R_{k(u)})$ is $1 + \mathrm{ql}(H^{n-1}(Y_u, \mathcal{O}_{Y_u})).$
\item[(ii)] 
The $D_{k(u)}$-module length of $H_{g_u}^1(R_{k(u)})$ is $1 + \dim_{\overline{k(u)}}(\overline{k(u)}\otimes_{k(u)}H^{n-1}(Y_u, \mathcal{O}_{Y_u}))^\ast,$ where $\overline{k(u)}$ is any algebraic closure of $k(u)$ and $(-)^\ast$ is the operation on $\overline{k(u)}[F]$-modules from Definition \ref{defi: *}.
\end{itemize}
\end{corollary}

\begin{proof} It is well-known that in this case the blow-up of the origin is a resolution $\pi'$ of the singularity and that the fibre at the origin is isomorphic to $Y.$ The result then immediately follows from Theorem \ref{thm: main} applied to $\pi'.$
\end{proof}
Thus the $D_k$-module length of the first local cohomology is closely related to ordinarity. Here is a simple example:

\begin{expl} \label{expl: elliptic curve}
Let $g$ be a rational cubic in three variables which is the equation of an elliptic curve $E$ in $\mathbb{P}^2_{\mathbb{Q}}.$ Then, for almost all primes $p,$ the $D_{R_{\mathbb{F}_p}}$-module length of $H^1_{g_p}(R_{\mathbb{F}_p})$ is $2$ if $E_p$ is ordinary and $1$ if $E_p$ is supersingular, where $E_p$ (resp. $g_p$) is the reduction of $E$ (resp. $g$) modulo $p.$
\end{expl}

\section{Comparison with Characteristic Zero}

Here, given a complex polynomial $g,$ we consider a holonomic $D_{\mathbb{C}}$-module $N_g$ whose length compares well to the $D_{k(u)}$-module length of $H_{g_u}^1(R_{k(u)}).$ 

\begin{definition} Let $g\in R_{\mathbb{C}}$ be a complex polynomial. Then $N_g$ is the left $D_{\mathbb{C}}$-submodule of the first local cohomology $D_{\mathbb{C}}$-module $H^1_g(R_{\mathbb{C}})$ generated by the class of $\frac{1}{g}.$
\end{definition}

The following is proved in \cite[Theorem 1.1]{bitoun2016d}.

\begin{theorem} \label{thm: char 0}
Let $g$ be a non-constant homogeneous complex polynomial in $n+1$ variables with an isolated singularity at the origin. Then, using the notation of Corollary \ref{cor: homogeneous}, the $D_{\mathbb{C}}$-module length of $N_g$ is $1+ \dim_\mathbb{C}H^{n-1}(Y,\mathcal{O}_Y).$

\end{theorem}

\begin{rmk} \label{rmk: comparison}
There is a ring of definition $B \subset \mathbb{C}$ of $Y$ such that, for all closed points $u$ of $\text{Spec}(B),$
$\dim_\mathbb{C}H^{n-1}(Y,\mathcal{O}_Y) =  \dim_{\overline{k(u)}}(\overline{k(u)}\otimes_{k(u)}H^{n-1}(Y_u, \mathcal{O}_{Y_u})).$ Hence by Corollary \ref{cor: homogeneous} and Theorem \ref{thm: char 0}, there is a ring of definition $B'\supset B$ of $Y$ such that for all closed points $u$ of $\text{Spec}(B'),$ if the Frobenius $F$ acts bijectively on $\overline{k(u)}\otimes_{k(u)}H^{n-1}(Y_u, \mathcal{O}_{Y_u}),$ then the length of $N_g$ is equal to the $D_{k(u)}$-module length of $H_{g_u}^1(R_{k(u)}).$ Indeed in that case, $\dim_\mathbb{C}H^{n-1}(Y,\mathcal{O}_Y) =  \dim_{\overline{k(u)}}(\overline{k(u)}\otimes_{k(u)}H^{n-1}(Y_u, \mathcal{O}_{Y_u}))^\ast.$ This property of the Frobenius is called weak ordinarity and is expected to hold for a dense set of closed points of $\Spec(B'),$ see \cite[Conjecture 1.1]{MR2863367}.
\end{rmk}

We would like to put forward the following questions:

\begin{q} \label{q: generalisation}
Let $g$ be a non-constant complex polynomial in $n+1$ variables. Is there a unitary finitely generated subring $B\subset \mathbb{C}$ containing the coefficients of $g$ such that:

\begin{enumerate}
\item \label{item: inequality}
For all closed points $u\in \Spec(B),$ $\lg_{D_{k(u)}}(H_{g_u}^1(R_{k(u)}))\leq \lg_{D_{\mathbb{C}}}(N_g)$? 
\item \label{item: ordinary}
There is a dense set of closed points of $\Spec(B)$ for which $\lg_{D_{k(u)}}(H_{g_u}^1(R_{k(u)}))= \lg_{D_{\mathbb{C}}}(N_g)$? 
\end{enumerate}
\end{q}

As explained in Remark \ref{rmk: comparison}, for $g$ homogeneous with an isolated singularity and $n\geq2,$ the first part of Question \ref{q: generalisation} has a positive answer. In the same case, the second part has a positive answer as well, if the weak ordinarity conjecture of \cite[Conjecture 1.1]{MR2863367} is satisfied by $Y.$ We finally note that by Theorem \ref{thm: main}, \cite[Conjecture 1.4]{bitoun2016d} (which is equivalent to \cite[Conjecture 3.8]{MR3283930}) implies a positive answer to the first part of Question \ref{q: generalisation}, for $g$ (not necessarily homogeneous) with an isolated singularity at the origin and $n\geq2.$

\section{Funding}

This work was supported by the Engineering and Physical Sciences Research Council [EP/L005190/1].

It is my pleasure to thank Manuel Blickle and Gennady Lyubeznik for interesting correspondence regarding the length of the positive characteristic first local cohomology module in the homogeneous case. I am also very grateful to Karl Schwede for explaining to me a proof of Lemma \ref{lemm: tight closure}. Finally, many thanks go to Johannes Nicaise and Travis Schedler for communicating to me proofs of Lemma \ref{lemm: cd}, as well as to the referee for pointing out that a proof was needed in the first place and for suggesting a simplification in the final argument.   

\bibliographystyle{plain}
\bibliography{bibfilex}

\end{document}